\documentclass[12pt]{amsart}

\usepackage{amssymb}

\usepackage{enumitem}

\usepackage{graphicx}

\makeatletter
\@namedef{subjclassname@2020}{%
  \textup{2020} Mathematics Subject Classification}
\makeatother

\usepackage[T1]{fontenc}

\newtheorem{theorem}{Theorem}[section]
\newtheorem{corollary}[theorem]{Corollary}
\newtheorem{lemma}[theorem]{Lemma}
\newtheorem{proposition}[theorem]{Proposition}

\theoremstyle{definition}
\newtheorem{definition}[theorem]{Definition}
\newtheorem{remark}[theorem]{Remark}
\newtheorem{example}[theorem]{Example}

\newtheorem*{xrem}{Remark}

\numberwithin{equation}{section}

\begin{document}


\baselineskip=17pt


\title[On $T$-orthogonality in Banach spaces]{On $T$-orthogonality in Banach spaces}

\author[Sain]{Debmalya Sain}
\address{Departamento de Analisis Matematico\\ Universidad de Granada\\Spain}
\email{saindebmalya@gmail.com}

\author[Ghosh]{Souvik Ghosh}
\address{Department of Mathematics\\ Jadavpur University\\ Kolkata 700032\\ West Bengal\\ INDIA}
\email{sghosh0019@gmail.com}

\author[Paul]{Kallol Paul}
\address{Department of Mathematics\\ Jadavpur University\\ Kolkata 700032\\ West Bengal\\ INDIA}
\email{kalloldada@gmail.com; kallol.paul@jadavpuruniversity.in}

\date{}

\begin{abstract}
	Let $\mathbb{X}$ be a Banach space and let $\mathbb{X}^*$ be the dual space of $\mathbb{X}.$  For $x,y \in \mathbb{X},$ $ x$ is said to be $T$-orthogonal to $y$ if $Tx(y) =0,$  where $T$ is a bounded linear operator from $\mathbb{X}$ to $\mathbb{X}^*.$ 
    We  study the notion of $T$-orthogonality in a Banach space and investigate its relation with the various geometric properties, like strict convexity, smoothness, reflexivity of the space.  We explore the notions of left and right symmetric elements w.r.t. the notion of $T$-orthogonality. We characterize bounded linear operators on $\mathbb{X}$ preserving $T$-orthogonality. Finally we characterize Hilbert spaces among all Banach spaces using $T$-orthogonality.
\end{abstract}



\maketitle

\section{Introduction}
The orthogonality notions in a Banach space play very important role in  the geometry of the underlying space. Recent literatures suggest that Birkhoff-James orthogonality and isosceles orthogonality contributed a lot towards better understanding of the geometry of Banach spaces. Needless to say all these notions of orthogonality coincide with usual notion of orthogonality in an Euclidean space. A new concept of orthogonality in the form of $T$-orthogonality in a Hausdorff topological vector space was introduced by the authors \cite{SK}  in which they have studied $T$-orthogonality additive functionals both when $T$-orthogonality is symmetric and not symmetric. Thereafter not much study has been done using $T$-orthogonality. Our main aim in this article is to study the notion of $T$-orthogonality in the setting of Banach space and explore its relation with the geometric properties of the Banach space. Before proceeding further we introduce the requisite notations and terminologies.

\smallskip

Let $\mathbb{X}$ be a Banach  space over the complex field $\mathbb{C}$ and let $\mathbb{X}^*$ denotes the dual space of $\mathbb{X}.$  For a complex number $\lambda,$ the real and imaginary part of $\lambda$ is denoted by $ \Re \, \lambda $ and $ \Im \, \lambda$ respectively.  Suppose $\mathbb{L}(\mathbb{X}, \mathbb{X}^*)$  denotes the collection of all bounded linear operators from $\mathbb{X}$ to $\mathbb{X}^*$ and $\mathbb{L}(\mathbb{X}) $ denotes the collection of all bounded linear operators on $\mathbb{X}.$ Let $ S_{\mathbb{X}}$ and $ B _{\mathbb{X}}$ denote the unit sphere and unit ball of the space  $\mathbb{X}$ i.e., $ S_{\mathbb{X}} = \{ x \in \mathbb{X}: \|x\|=1 \} $ and $ B_{\mathbb{X}} = \{ x \in \mathbb{X}: \|x\|\leq 1 \}. $ The space $\mathbb{X}$  is said to be  strictly convex if the unit sphere $ S_{\mathbb{X}} $ does not contain a non-trivial line segment i.e., if, $  (1-t)x + ty \in S_{\mathbb{X}} $ for some  $t \in (0,1) $  and $ x, y \in S_{\mathbb{X}},$ then $ x=y.$  An element $ x \neq 0 $ is said to be a smooth point of the space $\mathbb{X}$ if there exists a unique functional $ f \in \mathbb{X}^*$ such that $ f(x) = \|f\| \|x\|$ and $ \|f \| = \|x\|.$ The space $\mathbb{X}$ is said to be smooth if every non-zero element of the space $\mathbb{X}$ is a smooth point. Following \cite{SK}, $T$-orthogonality in a Banach space $\mathbb{X}$ can be defined as follows.
\begin{definition}
	Let $\mathbb{X}$ be a Banach  space and let $T \in \mathbb{L}(\mathbb{X}, \mathbb{X}^*).$  For $x, y \in \mathbb{X},$ \emph{$x$ is $T$-orthogonal to $y, $}  written as $ x \perp_T y,$ if $Tx(y) = 0.$  We write $Tx(y) = (Tx,y).$ 
\end{definition}
$T$-orthogonality is said to be \emph{symmetric} if $(Tx,y)=0 $ implies $(Ty,x) =0,$  for all $x, y \in \mathbb{X}.$ 
The operator $T$ is said to be \emph{symmetric} if $(Tx,y) = (Ty,x), $ for all $ x, y \in \mathbb{X}.$   An element $x \in \mathbb{X}$ is said to be \emph{$T$-isotropic} if $(Tx,x) =0.$  We here introduce the following definitions.
\begin{definition}
	Let $\mathbb{X}$ be a  Banach  space and let $T \in \mathbb{L}(\mathbb{X}, \mathbb{X}^*).$  Let $ \theta \in [0, 2 \pi).$ For $x, y \in \mathbb{X}, $ \emph{$ x$ is said to be $T$-orthogonal to $y$ in the direction of $\theta,$} written as $ x \perp_{T_\theta} y, $  if $ \cos \theta \, \Re (Tx,y) + \sin \theta \, \Im (Tx,y) = 0.$   For brevity, we write $ (T_\theta x, y) = \cos \theta \, \Re (Tx,y) + \sin \theta \, \Im (Tx,y).$  
	We say $y \in x_{T_{\theta}}^+$ if $ (T_\theta x, y) \geq 0$  and $ y \in x_{T_\theta}^-$ if $ (T_\theta x, y) \leq 0.$ In case $\mathbb{X}$ is a real Banach space we write $y \in x_T^+$ if $ (Tx, y) \geq 0$ and $ y \in x_T^-$ if $ (Tx, y) \leq 0. $
\end{definition}

Next we introduce the notion of  left symmetric and right symmetric points with respect to $T$-orthogonality.  
\begin{definition}
	Let $\mathbb{X}$ be a Banach  space and let $T \in \mathbb{L}(\mathbb{X}, \mathbb{X}^*).$  We say that \emph{$T$-orthogonality is left symmetric at $x \in \mathbb{X},$} if $ x \perp_T y $ implies $ y \perp_T x $ for each $ y \in \mathbb{X}.$  In short, we write `$\perp_T$ is left symmetric at $x$'. \\
	Further for each $ \theta \in [0, 2 \pi), $ we say that  \emph{$T$-orthogonality is left symmetric at $x$ in the direction of $\theta$} if $ x \perp_{T_\theta} y $ implies $ y \perp_{T_\theta}x$ for each $ y \in \mathbb{X}.$  In short, we write `$ \perp_{T_\theta}$ is left symmetric at $x$'. 
\end{definition}

\begin{definition}
	Let $\mathbb{X}$ be a  Banach  space and let $T \in \mathbb{L}(\mathbb{X}, \mathbb{X}^*).$   We say that \emph{$T$-orthogonality is right  symmetric at $x \in \mathbb{X},$} if $ y \perp_T x $ implies $ x \perp_T y $ for each $ y \in \mathbb{X}.$  In short, we write `$\perp_T$ is right symmetric at $x$'. \\
	Further for each $ \theta \in [0, 2 \pi), $ we say  that  \emph{$T$-orthogonality is right  symmetric at $x$ in the direction of $\theta$} if $ y \perp_{T_\theta} x $ implies $ x \perp_{T_\theta}y$ for each $ y \in \mathbb{X}.$  In short, we write `$ \perp_{T_\theta}$ is right symmetric at $x$'. 
\end{definition}

We need the notion of Birkhoff-James orthogonality \cite{B,Ja}  to study the geometric properties of the Banach space using $T$-orthogonality. For $x, y \in \mathbb{X}$, $ x $ is said to be Birkhoff-James orthogonal to $y$, if $ \| x + \lambda y \| \geq \|x\| $ for all $\lambda  \in \mathbb{C}.$ The marvelous paper by James \cite{J} nicely characterizes the properties like smoothness, strict convexity, reflexivity using Birkhoff-James orthogonality. 

In this article we study the symmetric properties of $\perp_T$ and $\perp_{T_\theta}$ and explore the relations between them. We also characterize the geometric properties of the space like smoothness and strict convexity using the notion of $T$-orthogonality. We also obtain necessary and sufficient condition for an operator $ A \in \mathbb{L}(\mathbb{X})$ preserving  $T$-orthogonality in terms of newly introduced notion of $T$-isometry.  An operator  $ A \in \mathbb{L}(\mathbb{X})$  is  said to be $T$-isometry if it satisfies $T = A^*TA,$ where $A^*$ is the adjoint of $A.$  Finally we characterize Hilbert spaces among Banach spaces using $T$-orthogonality.

\section{Main results}

We begin this section by noting the following basic properties in the form of an easy proposition. The proofs are omitted as they can be obtained by applying elementary geometric and algebraic reasoning. 
\begin{proposition}\label{basic}
	Let $\mathbb{X}$ be a  Banach  space and let $T \in \mathbb{L}(\mathbb{X}, \mathbb{X}^*).$
	\begin{enumerate}[label=\upshape(\roman*), leftmargin=*, widest=iii]
		\item Given any $ x, y \in \mathbb{X},$ there exists $ \theta \in [0, 2 \pi)$ such that $ x \perp_{T_\theta} y.$
		\item For $x, y \in \mathbb{X}, $  $x \perp_T y$  if and only if   $ x \perp_{T_\theta}y $ and $ x \perp_{T_\phi} y $  for some $\theta$ and $ \phi,$ where $ \theta - \phi \neq 0, \pi.$ 
		\item  For $x, y \in \mathbb{X}, $ $ x \perp_T y $ implies that $ x \perp_{T_\theta} y ,$ for each $\theta.$   However, there are operators $T$ for which  $ x \perp_{T_\theta} y, $ for some particular $\theta,$ whereas $ x \not \perp_T y.$
		\item For each $x,y \in \mathbb{X}$ and $ \theta, \phi \in [0,2 \pi), $ we have, $ (T_\theta x, e^{i \phi}y) = (T_{\theta-\phi}x,y) = (T_\theta (e^{i \phi}x), y) $  and so $ x \perp_{T_\theta} e^{i \phi}y \Leftrightarrow x \perp_{T_{\theta - \phi}} y \Leftrightarrow e^{i \phi}x \perp_{T_\theta} y .$ 
		\item $ \perp_T$ is  right  additive i.e., if $ x \perp_T y, x \perp_T z $ then $ x \perp_T (y +z) .$ 
		\item $ \perp_T$ is left additive i.e., if $x \perp_T z, y \perp_T z$ then $ (x+y) \perp_T z.$ 
	\end{enumerate}
\end{proposition}
An elementary example of $T \in \mathbb{L}(\mathbb{X}, \mathbb{X}^*)$  such that  $ x \perp_{T_\theta} y, $ for some particular $\theta,$ whereas $ x  \perp_T y$  does not hold is as follows: Consider the two-dimensional complex Hilbert space $ \mathbb{C}^2.$ 
Let $T : \mathbb{C}^2 \longrightarrow  (\mathbb{C}^2)^*$ be defined as $ T(1, 0) =  7i (1, 0)^* + (0, 1)^*$ and $T(0, 1) = 2(1, 0)^* + 3i (0, 1)^*,$ where $ (1, 0)^* (z_1, z_2) = z_1$ and $ (0, 1)^*(z_1, z_2) = z_2,$ for all $(z_1, z_2) \in \mathbb{C}^2.$ Then $(0,1) \perp_{T_{\pi/4}} (1/2,-1/3) $ but $(0,1) \not \perp_T (1/2,-1/3).$ \\
Next we prove  the following proposition.

\begin{proposition}\label{symm}
	Let $\mathbb{X}$ be a complex Banach space and let $T \in L(\mathbb{X}, \mathbb{X}^*).$ Suppose that $x \in \mathbb{X}$ and $Tx \neq 0.$ If $\perp_T$ is left symmetric at $x$ then there exists a scalar $\lambda$ such that $(Tx, y) = \lambda (Ty, x)$ for all $ y \in \mathbb{X}.$ 
\end{proposition}

\begin{proof}
	
	Clearly  there exists $ 0 \neq z \in \mathbb{X} $ such that $ (Tx,z) \neq 0.$  Then  $(Tx,z) = \lambda(Tz,x)$ for some scalar $\lambda.$ We claim that $(Tx,y) = \lambda(Ty,x),$ for all $y \in \mathbb{X}.$ As before we can write $ \mathbb{X} = Ker \, Tx \oplus \langle \{z\} \rangle.$ Let $ y \in \mathbb{X}.$ Then $ y = \alpha z + v ,$ for some scalar $ \alpha $ and $ v \in  Ker \, Tx.$  Now $(Tx,v) =0$ implies $ (Tv,x) =0 $ by left symmetricity of $\perp_T $ at $x.$ Then $ (Tx,y) = \alpha (Tx,z) + (Tx,v) = \alpha \lambda(Tz,x) + \lambda(Tv,x) = \lambda (Ty,x) .$ 
	This completes the proof.
\end{proof}

\begin{theorem} \label{direction}
	
	Let $\mathbb{X}$ be a complex Banach space and let $T \in \mathbb{L}(\mathbb{X}, \mathbb{X}^*).$ Suppose that $x \in \mathbb{X} $ and $\theta \in [0, 2\pi).$ If $\perp_{T_\theta}$ is left symmetric at $x$ then $\perp_T$ is left symmetric at $x.$
\end{theorem}

\begin{proof} Let $ \gamma \in [0, 2 \pi) .$ We first claim that $\perp_{T_\gamma}$ is left symmetric at $x$.  Let $ x \perp_{T _\gamma} y.$ Then by Proposition \ref{basic} (iv), we get $ 0 = ( T_\gamma x, y ) = (T_\theta x, e^{i(\theta- \gamma)} y) .$  So we get $( T_\theta (e^{i(\theta - \gamma)}y, x ) = 0,$ by left symmetricity of $T_\theta$ at $x.$ Again by Proposition \ref{basic} (iv), we get, $ (T_\gamma y, x) = ( T_\theta (e^{i(\theta - \gamma)}y, x ) = 0,$ and hence $  y \perp_{T_\gamma} x. $ Thus $ \perp_{T_\gamma}$ is left symmetric at $x.$  Now we show that $\perp_T$ is left symmetric at $x.$ Let $ x \perp_T y.$ Then  $x \perp_{T_\theta} y $ for any $\theta \in [0, 2 \pi)$ and by left symmetricity of $\perp{T_\theta}$ we get, $ y \perp_{T_\theta} x .$ Then by Proposition \ref{basic} (ii),  we get,  $  y \perp_T x.$  
\end{proof} 

In this following example we show that the converse is not true in general.
\begin{example} \label{example}
	Consider the two-dimensional complex Hilbert space $ \mathbb{C}^2.$ Every element of $ \mathbb{C}^2$ can be written as $ (z_1, z_2),$ where $ z_1, z_2 \in \mathbb{C}.$ 
	Let $T : \mathbb{C}^2 \longrightarrow  (\mathbb{C}^2)^*$ be defined as $ T(1, 0) = i (0, 1)^*$ and $T(0, 1) = (1, 0)^* + 3i (0, 1)^*,$ where $ (1, 0)^* (z_1, z_2) = z_1$ and $ (0, 1)^*(z_1, z_2) = z_2,$ for all $(z_1, z_2) \in \mathbb{C}^2.$ Clearly, $(1, 0)$ is an isotropic vector and $\perp_T$ is both left and right symmetric at $(1, 0).$ Now for $ (1, i) \in \mathbb{C}^2$ it is straightforward to see that $ (1, 0) \perp_{T_{\pi/2}} (1, i)  $ whereas $ (1, i) \not \perp_{T_{\pi/2}} (1, 0).$ Therefore, $ \perp_{T_{\pi/2}}$ is not left symmetric at $(1, 0).$
\end{example}

\begin{xrem}
	The above example shows that for isotropic vector $x$ the converse of Theorem \ref{direction} may not hold true. However, if $x$ is a nonisotropic vector then following Proposition \ref{symm}, we get a scalar $\lambda $ such that $(Tx, y) = \lambda (Ty, x)$  for all $y \in \mathbb{X}.$  In case $ \lambda $ is real,  then  $\perp_T$ is left symmetric at $x$ will imply that $\perp_{T_\theta}$ is left symmetric at $x$ for all $\theta \in [0, 2\pi).$  
\end{xrem}

\begin{theorem} \label{theorem}
	Let  $\mathbb{X}$ be a complex Banach space and let $T \in \mathbb{L}(\mathbb{X}, \mathbb{X}^*).$ Then the followings hold:
	\begin{enumerate}[label=\upshape(\roman*), leftmargin=*, widest=iii]
	 \item For any nonisotropic vector $x \in \mathbb{X},$ $\perp_T$ is left symmetric at $x$ if and only if $\perp_T$ is right symmetric at $x.$
	 \item If $T$ is bijective then for any nonzero isotropic vector $x\in \mathbb{X},$ $\perp_T$ is left symmetric at $x$ if and only if $\perp_T$ is right symmetric at $x.$
	\end{enumerate} 
\end{theorem}
\begin{proof}
	(i) We have $Tx(x) \neq 0.$  Assume $\perp_T$ is left symmetric at $x.$ We show that $\perp_T$ is right symmetric at $x.$
	If possible let $ y \perp_T x $ but $ x \not \perp_T y.$ Then $Tx(y) \neq 0$  and so $ Tx(y) = \alpha Tx(x) $ for some scalar $\alpha.$ This implies $Tx(y - \alpha x) =0$ so that $ x \perp_T (y -\alpha x).$ Since $\perp_T $ is left symmetric at $x$ so we get $ (y - \alpha x) \perp_T x $. Thus $ T (y- \alpha x) (x) =0  \Rightarrow Ty(x) = \alpha Tx(x) = Tx(y). $  This leads to a   contradiction. Thus
	$\perp_T$ is right symmetric at $x.$ Analogously one can show that if  $\perp_T$ is right symmetric at $x$ then $\perp_T$ is left  symmetric at $x.$
	
	\smallskip
	(ii) Assume $\perp_T$ is left symmetric at $x$ where $ x \neq 0$ and $ (Tx,x) =0.$ If possible let there exist $ z \in \mathbb{X}$ such that $ z \perp_T x $ but $ x \not \perp z.$  Then $ z \not \in  Ker \,Tx $ and $Tx $ being a linear functional on $\mathbb{X},$ we get $ \mathbb{X} = Ker \, Tx \oplus \langle \{z\} \rangle.$  Let $ w \in \mathbb{X}.$ Then $w = \alpha z + v $ where $ v \in \, Ker \,Tx $ and $ \alpha $ is a scalar. So $Tx(v) =0 $ implies $x \perp_T v $ and hence  $ v \perp_T x ,$ by left symmetricity of $\perp_T$ at $x.$  Thus $Tv(x) =0 .$  This shows that $ Im \, T \subset \{x\}^0$ where $ \{x\}^0$ is the annihilator of $x.$ Then by bijectivity of $T$ we get, $ \mathbb{X}^* = \{x\}^0 $  which contradicts the fact that $x \neq 0.$ Thus $\perp_T$ is right symmetric at $x.$ \\
	Next we assume that  $\perp_T$ is right symmetric at $x$ where $ x \neq 0$ and $ (Tx,x) =0.$ If possible let there exist $ z \in \mathbb{X}$ such that $ x \perp_T z $ but $ z \not \perp x.$  Then $Tz(x) \neq 0 $ and so $ Tz \not \in \{x\}^0.$  Now $ \mathbb{X}^* = \{x\}^0 \oplus \langle \{Tz\}\rangle .$  Let $ u \in \mathbb{X}.$ Then $ Tu \in \mathbb{X}^*$ and so $ Tu = g + \alpha Tz$ for some $ g \in \{x\}^0$ and for some scalar $\alpha.$ Since $T $ is bijective so there exists unique $v \in \mathbb{X}$ such that $Tv=g$ and $ Tv(x) =0$. This implies $ v \perp_T x $ and by right symmetricity of $\perp_T$ we get, $ x \perp_T v.$  Now  $ T u = T v + \alpha Tz $ implies $ u = v  + \alpha z.$  Then $Tx(u) = Tx(v) + Tx(z) =0 $ so that $ x \perp_T u.$ Thus for each $u \in \mathbb{X}$, $Tx(u) = 0$ and so $ Tx=0$. This implies $x=0$, by injectivity of $T$. This is a contradiction. Thus $\perp_T$ is left symmetric at $x.$ 
\end{proof}

In the following example we show that in Theorem \ref{theorem} (ii), the `bijectivity' condition on $T$  can not be omitted.

\begin{example}
	Let $ \ell_2^2$ be the 2-dimensional real Hilbert space.
	Consider $ T : \ell_2^2 \longrightarrow (\ell_2^2)^*$ is defined as $T(1, 0) = (1, -1)^*,$ where $ (1, -1)^*(x, y) = x-y.$ and $ T(0,1) = (2, -2)^*,$ where $ (2, -2)^*(x, y) = 2x -2y.$ Clearly, $ T$ is not bijective and $ (1, 1)$ is a isotropic vector. It is easy to check that $\perp_T$ is left symmetric at $(1, 1).$ On the other hand, observe that $ (1, 0) \perp_T (1, 1)$ whereas $(1, 1) \not\perp_T (1, 0).$ This implies that $\perp_T$ is not right symmetric at $(1, 1).$  
\end{example}

Our next lemma deals with an interesting observation about nonisotropic vectors.

\begin{lemma} \label{lemma}
	Let $\mathbb{X}$ be a complex Banach space and let $T \in \mathbb{L}(\mathbb{X}, \mathbb{X}^*).$ Suppose that $ x \in \mathbb{X}$ is a nonisotropic vector at which $\perp_T$ is left (right) symmetric. Then $(Tx, y) = (Ty, x),$ for all $y \in \mathbb{X}.$
\end{lemma}

\begin{proof}
	Since $\perp_T$ is left symmetric at $x,$ $(Tx, y)= 0$ implies $(Ty, x)= 0.$ Suppose that $(Tx, y) \neq 0.$ It is easy to observe that for $\alpha \in \mathbb{C},$ $(Tx, \alpha x+y) = \alpha (Tx, x) + (Tx, y).$ Choosing $ \alpha = \frac{-(Tx, y)}{(Tx, x)}$ we obtain that $(Tx, \alpha x+y)=0.$ As $\perp_T$ is left symmetric at $x,$ \[(T(\alpha x+y), x)=0 \implies \alpha (Tx, x) + (Ty, x)=0 \implies (Tx, y)= (Ty, x).\]
	This proves our lemma.
\end{proof}

Before we proceed to our next theorem we would like to note that the above lemma is not necessarily true for isotropic vector. The following example illustrates our claim.

\begin{example}
	Let $ T : \ell_2^2 \longrightarrow (\ell_2^2)^*$ be defined as $ T(1, 0) = (0, 1)^*$ and $ T(0, 1) = (\frac{1}{\sqrt{2}}, \frac{1}{\sqrt{2}})^*,$ where $\ell_2^2$ is 2-dimensional real Hilbert space. It is easy to verify that $ (1, 0)$ is an isotropic vector and $ \perp_T$ is both left and right symmetric at $(1, 0).$ For any $ (\alpha, \beta), (x, y) \in \ell_2^2,$ $(T(\alpha, \beta), (x, y)) = \frac{\beta}{\sqrt{2}}x +(\alpha + \frac{\beta}{\sqrt{2}})y.$ Then $ (T(\alpha, \beta), (1, 0)) = \frac{\beta}{\sqrt{2}}$ and $(T(1, 0), (\alpha, \beta)) = \beta.$ Therefore, whenever $ \beta \neq 0,$ we have $ (T(\alpha, \beta), (1, 0)) \neq (T(1, 0), (\alpha, \beta)).$
\end{example}

In our next theorem we characterize $T \in \mathbb{L}(\mathbb{X}, \mathbb{X}^*)$ which are symmetric. Note that the following theorem is a clear improvement of \cite[Lemma 1]{SK}.

\begin{theorem} \label{nonisotropic; isotropic}
	Let $\mathbb{X}$ be a complex Banach space and let $T \in \mathbb{L}(\mathbb{X}, \mathbb{X}^*)$ be nonzero. Then $T$ is symmetric if and only if the followings hold true:
	\begin{enumerate}[label=\upshape(\roman*), leftmargin=*, widest=iii]
		\item There exists a nonisotropic vector $x \in \mathbb{X}.$
		\item $\perp_T$ is left(right) symmetric at $y,$ where $y $ is a nonisotropic vector in $\mathbb{X}.$
	\end{enumerate}
	
\end{theorem}

\begin{proof}
	First we prove the sufficient part of the theorem. Suppose that $z, w \in \mathbb{X}$ are two arbitrary elements in $\mathbb{X}.$ If one of them is nonisotropic then from Lemma \ref{lemma}, we get $(Tz, w) = (Tw, z).$ Let $ z, w$ both be isotropic. Then from \cite[Lemma 1]{SK} we have either $x +z$ or $x - z$ is nonisotropic. Without loss of generality we assume that $x+z$ is nonisotropic. Therefore,
	again using Lemma \ref{lemma}, we obtain 
	\[ (T(x+z), w) = (Tw, x+z) \implies (Tx, w)+(Tz, w) = (Tw, x) + (Tw, z). \] Since $x $ is nonisotropic, $(Tx, w)= (Tw, x).$ Therefore, $(Tz, w)=(Tw, z).$ Hence $T$ is symmetric.\\
	Let us now prove the necessary part. Clearly (ii) holds  if $T$ is symmetric. We just prove $(i).$ Suppose on the contrary that every vector in $\mathbb{X} $ is isotropic, i.e., $(Tx, x) = 0,$ for all $x \in \mathbb{X}.$ Take $u, v \in \mathbb{X}.$ Then $(T(u+v), u+v) = 0$ which implies $(Tu, v) = -(Tv, u).$ As $T$ is symmetric, it can be readily seen that $(Tu, v) = 0.$ As $u, v$ are chosen arbitrarily, $T $ is zero. This contradiction proves the necessary part of the theorem.
	
\end{proof}

\begin{theorem} \label{left symmetric}
	
	Let $\mathbb{X}$ be a complex Banach space and let $T\in \mathbb{L}(\mathbb{X}, \mathbb{X}^*).$ Suppose that $x \in \mathbb{X}.$ Then $\perp_T$ is left symmetric at $x$ if and only if there exists $\phi_0 \in [0, 2\pi)$ such that the followings hold true:
	\[ y \in x_{T_\theta}^+ \implies x \in y_{T_{\theta - \phi_0}}^+ \,  \mbox{and}~ \,
	y \in x_{T_\theta}^- \implies x \in y_{T_{\theta -\phi_0}}^-,\, \mbox{for every }  \, \theta \in [0, 2\pi).\]
	
\end{theorem}

\begin{proof}
	Let us first prove the sufficient part of the theorem. Let $y \in \mathbb{X}$ be such that $(Tx, y) = 0.$ This implies that $y \in x_{T_\theta}^+ \cap x_{T_\theta}^-,$ for all $ \theta \in [0, 2\pi).$ Then there exists $ \phi_0 \in [0, 2\pi)$ such that $x \in y_{T_{\theta - \phi_0}}^+ \cap y_{T_{\theta - \phi_0}}^-,$ i.e., $x \in y_{T_\theta}^+ \cap y_{T_\theta}^-.$ Therefore, $(Ty, x) = 0.$ This proves the sufficient part. 
	We next prove the necessary part.  If $Tx = 0 $ then the result holds trivially. Assume $Tx \neq 0.$ 
	Then by Proposition \ref{symm},  there exists a scalar $\lambda $ such that $(Tx,u) = \lambda(Tu,x),$ for all $u \in \mathbb{X}.$  Let $ y \in x^+_{T_\theta}.$ Then $ ( T_\theta x, y ) \geq 0 $ and by using Proposition \ref{basic} (iv) we get, $   ( T_{\theta- \phi_0}y,x)  \geq 0,$  where $ \lambda = \mid \lambda \mid e^{i \phi_0}.$ This shows that $ x \in y^+_{T_{\theta - \phi_0}}.$ Similarly we can show that $  y \in x_{T_\theta}^- \implies x \in y_{T_{\theta -\phi_0 }}^-.$ This completes the necessary part.
	
\end{proof}

The following corollary is an immediate consequence of the last theorem.

\begin{corollary}
	Let $\mathbb{X}$ be a real Banach space and let $T\in \mathbb{L}(\mathbb{X}, \mathbb{X}^*).$ Suppose that $x \in \mathbb{X}.$  Then $\perp_T$ is left symmetric at $x$ if and only if either of the following hold true:
	\begin{enumerate}[label=\upshape(\roman*), leftmargin=*, widest=iii]
	\item $y \in x_T^+ \implies x \in y_T^+$ and $y \in x_T^- \implies x \in y_T^-.$
	\item $y \in x_T^+ \implies x \in y_T^-$ and $y \in x_T^- \implies x \in y_T^+.$	
	\end{enumerate}

\end{corollary} 

\begin{proof}
	We first prove the necessary part.  If $Tx =0 $ then for each $ y \in \mathbb{X}, \, ( Tx,y) =0 $ and by left symmetricity of $\perp_T$ at $x$ we get $(Ty,x) =0.$ Thus both (i) and (ii) hold trivially. Next assume $Tx \neq 0.$ Then by Proposition \ref{symm}, there exists a real scalar $\lambda$ such that $ (Tx, y) = \lambda(Ty, x), $ for all $y \in \mathbb{X}.$  Thus (i) or (ii) holds accordingly as  $ \lambda \geq 0$ or $\lambda \leq 0.$  This completes the proof of necessary part. 
	The sufficient part follows easily from the definition of left symmetricity of $\perp_T$ at $x$. 
\end{proof}

We  say ``$ \perp_T$ is left symmetric in $\mathbb{X}$''  if $ \perp_T$ is left symmetric at $x,$ for each $x \in \mathbb{X}.$ We note the following remark in this regard.

\begin{remark}
	Let $\mathbb{X}$ be a complex Banach space and let $T\in \mathbb{L}(\mathbb{X}, \mathbb{X}^*).$ Using Theorem \ref{nonisotropic; isotropic} together with Theorem \ref{left symmetric}, it can be concluded that $ \perp_T$ is left symmetric in $ \mathbb{X}$ if and only if either of the following holds true:
	\begin{enumerate}[label=\upshape(\roman*), leftmargin=*, widest=iii]
		\item  $y \in x_{T_{\theta}}^+ \implies x \in y_{T_{\theta}}^+$ and $y \in x_{T_\theta}^- \implies x \in y_{T_\theta}^-.$
	\item $y \in x_{T_\theta}^+ \implies x \in y_{T_\theta}^-$ and $y \in x_{T_\theta}^- \implies x \in y_{T_\theta}^+,$ for every $ \theta \in [0, 2\pi).$ 
	\end{enumerate}
	On a more precise note we can observe that if there exists a nonisotropic vector in $ \mathbb{X}$ then $ \perp_T$ is left symmetric in $\mathbb{X}$ if and only if $y \in x_{T_{\theta}}^+ \implies x \in y_{T_{\theta}}^+$ and $y \in x_{T_\theta}^- \implies x \in y_{T_\theta}^-,$ whereas if all the vectors in $\mathbb{X}$ are isotropic then $\perp_T$ is left symmetric in $\mathbb{X}$ if and only if $y \in x_{T_\theta}^+ \implies x \in y_{T_\theta}^-$ and $y \in x_{T_\theta}^- \implies x \in y_{T_\theta}^+,$ for every $ \theta \in [0, 2\pi),$ for every $ \theta \in [0, 2 \pi).$	
	
\end{remark}

Let us note that $ x^{\perp} = \{y \in \mathbb{X} : x \perp_B y\}$ and  $x^{\perp_T} =  \{y \in \mathbb{X} : x \perp_T y\}. $

\begin{proposition} \label{prop}
	Let $ \mathbb{X}$ be a complex Banach space and let $ T \in \mathbb{L}(\mathbb{X}, \mathbb{X}^*)$ be bijective. Then $ \mathbb{X}$ is reflexive if and only if for any nonzero $ x \in \mathbb{X},$ there exist $ z \in \mathbb{X}$ such that $ z \perp_B x^{\perp_T}.$
\end{proposition}

\begin{proof}
	It is easy to observe that $ x ^{\perp_T} $ is a hyperspace of $ \mathbb{X}.$ Since $ \mathbb{X} $ is reflexive, there exist a $ z \in \mathbb{X}$ such that $ z \perp_B x^{\perp_T}.$ Conversely, Let $ H$ be a hyperspace in $ \mathbb{X}.$ Consider $ f \in \mathbb{X}^*$ as the corresponding functional of $H$ such that $ ker~f = H.$ As $T$ is bijective, let $ T(w)= f,$ for some $ w\in \mathbb{X}.$ Clearly, $w^{\perp_T} = H.$ Then there exist  $ v \in \mathbb{X}$ such that $ v \perp_B H.$ Then from \cite{Ja} we conclude that  $ \mathbb{X}$ is reflexive. 
\end{proof}

\begin{theorem} \label{smooth; strict convex}
	Let $\mathbb{X}$ be a complex Banach space and let $T\in \mathbb{L}(\mathbb{X}, \mathbb{X}^*)$ be bijective. Then we have the followings:
	\begin{enumerate}[label=\upshape(\roman*), leftmargin=*, widest=iii]
	\item A nonzero element $x \in \mathbb{X}$ is smooth if and only if there exists a unique (upto scalar multiplication) $z \in \mathbb{X}$ such that $z^{\perp_T} = x^{\perp}.$ 
	 \item The space $\mathbb{X}$ is strictly convex if and only if for every $x \in \mathbb{X},$ there exists atmost one (upto scalar multiplication) $z \in S_\mathbb{X}$ such that $z \perp_B x^{\perp_T}.$ 
	\end{enumerate}

\end{theorem}

\begin{proof}
	$(i)$ Let $y \in x^{\perp}.$ Since $x$ is smooth, $J(x) = \{x^*\}$ and $x^*(y)=0.$ We note that $T$ is bijective and therefore, there exists $z \in \mathbb{X}$ such that $Tz = x^*.$ This implies that $z \perp_T y.$ Thus we obtain $ x^{\perp} \subset z^{\perp_T}.$ Observe that $ x^{\perp} = ker x^* = z ^{\perp_T} = (\lambda z)^{\perp_T}.$ Therefore, there exists a unique (upto scalar multiplication) $z \in \mathbb{X}$ such that $ z^{\perp_T} = x^{\perp}.$
	Conversely, suppose that $x$ is a nonzero element of $\mathbb{X}.$ Let $ x \perp_B u$ and $ x \perp_B v,$ for some $ u, v \in \mathbb{X}.$ Then there exists a unique $ z \in \mathbb{X}$ such that $ z \perp_T u$ and $ z \perp_T v.$ As $ Tz \in \mathbb{X}^*,$ $ z \perp_T u+v,$ which implies $ x \perp_B u+v.$ Thus from \cite{J}, we obtain that $ x $ is a smooth point.\\
	
	$(ii)$ Let $ \mathbb{X}$ be strictly convex and let $ x\in \mathbb{X}$ be a nonzero element. Then $ M_{Tx} = \emptyset$ or $ M_{Tx} = \{e^{i\theta}z\},$ where $ z \in S_{\mathbb{X}}$ and $ \theta \in [0, 2\pi).$ It is easy to verify that when $ M_{Tx} = \emptyset,$ there exists no such $ z \in S_{\mathbb{X}}$ such that $ z\perp_B x^{\perp_T}.$ Let $ M_{Tx} = \{e^{i\theta}z\}.$ This implies that $ |Tx(z)| = \|Tx\|,$ i.e., $ \mu \frac{Tx}{\|Tx\|} \in J(z),$ for some $\mu \in S^1.$ Thus for any $ y \in x^{\perp_T},$ we get $ z \perp_B y.$ In other words, $ z \perp_B x^{\perp_T}.$ Now if we assume that $ z' \in S_{\mathbb{X}}$ with $ z' \neq \gamma z,$ where $ \gamma \in S^1$ such that $ z' \perp_B x^{\perp_T} $ then one can easily show that $ z' \in M_{Tx}.$ This leads to a contradiction that $ \mathbb{X}$ is strictly convex.
	
	Conversely, Suppose on the contrary that $ \mathbb{X}$ is not strictly convex. Then there exists $ u, v (u \neq \lambda v) \in S_{\mathbb{X}}$ such that $ J(u) \cap J(v) \neq \emptyset,$ where $ \lambda \in S^1.$ Let us take $ f \in J(u) \cap J(v)$ and let $ Tw = f,$ for some $ w \in \mathbb{X}.$ Clearly, $ w^{\perp_T} = ker f.$ Therefore, it can be readily seen that $ u \perp_B w^{\perp_T}$ as well as $ v \perp_B w^{\perp_T}.$ This contradicts our hypothesis and proves our theorem.
\end{proof}

\begin{xrem}
	Let $ \mathbb{X}$ be a complex Banach space and let $ T \in \mathbb{L}(\mathbb{X}, \mathbb{X}^*)$ be bijective. From Theorem \ref{smooth; strict convex}(ii) and  Proposition \ref{prop}, we can conclude that if $\mathbb{X}$ is reflexive and strictly convex then for any $x \in \mathbb{X},$ there exists a unique (upto scalar multiplication) $z \in S_\mathbb{X}$ such that $ z \perp_B  x^{\perp_T}$ whereas if $ \mathbb{X}$ is non-reflexive and strictly convex then there exists a $ w \in \mathbb{X}$ such that for no $z \in \mathbb{X},$ $ z \perp_B w^{\perp_T}$ holds true.
\end{xrem}

For a given Banach space $\mathbb{X}$ it is a natural question to ask which operators preserves $T$-orthogonality. In \cite{Kold} it is proved that a linear operator on $\mathbb{X}$ preserves `$\perp_B$' if the operator is a positive multiple of an isometry.  We next characterize the operators preserving `$\perp_T$'.

\begin{theorem}
	Let $\mathbb{X}$ be a complex Banach space and let $T\in \mathbb{L}(\mathbb{X}, \mathbb{X}^*)$ be bijective.  Suppose that $ A \in \mathbb{L}(\mathbb{X}).$ Then for any $ x, y \in \mathbb{X},$ $ x \perp_T y \iff Ax \perp_T Ay$ if and only if $A$ is a scalar multiple of $T$-isometry.
\end{theorem}

\begin{proof}
	First we prove the necessary part of the theorem. Suppose that $ x \perp_T y \iff Ax \perp_T Ay,$ for all $x , y \in \mathbb{X}.$  This implies that for any nonzero $ x \in \mathbb{X},$ $ (Tx, y) = 0 \implies (TAx, Ay) = 0.$ This gives us $ (TAx, Ay) = 0 \implies (A^*TAx, y) =0.$ Therefore, $ y \in ker Tx$ implies $ y \in ker A^*TAx,$ for every $ y \in \mathbb{X}.$ Since $T$ is bijective, for all nonzero $x \in \mathbb{X},$  we obtain $ ker Tx = ker A^*TAx,$ i.e., $Tx = \beta_x A^*TAx,$ for some $ \beta_x \in \mathbb{C}.$  Now we claim that $\beta_x$ is a fixed constant. Suppose that $u, v \in \mathbb{X}$ be two arbitrary nonzero vector. It is easy to check that if $v = \alpha u,$ where $\alpha \in \mathbb{C}$ then $Tv = \beta_u A^*TAv. $ Let $ u, v$ be independent. Then
	\begin{equation*}
		\begin{aligned}
		Tu +Tv &=T(u+v)\\
	           &= \beta_{u+v} A^*TA(u+v) \\ 
		       &= \beta_{u+v} A^*TAu + \beta_{u+v} A^*TAv \\
		       &= \frac{\beta_{u+v}}{\beta_u} Tu + \frac{\beta_{u+v}}{\beta_v} Tv.
	\end{aligned}
	\end{equation*} 
	As $T$ is bijective and $u, v$ are linearly independent, $\beta_u = \beta_v = \beta_{u+v}.$  So our claim is established. Therefore, we obtain that $T = \beta A^*TA,$ as desired.\\
	To prove the sufficient part of the theorem let us assume that $ T = \lambda A^*TA,$ for some $\lambda \in \mathbb{C}.$ Then for any $ x, y \in \mathbb{X},$ with $x \perp_T y$ we get 
	\[
	(Tx, y) = 0 
	\iff (\lambda A^*TAx, y) = 0 
	\iff \lambda (TAx, Ay) = 0
	\iff Ax \perp_T Ay.
	\]
	Hence our theorem is proved.
\end{proof}

Now in case of real Banach spaces with $dim\mathbb{X} \geq 3,$ we are going to show that $ T$-orthogonality and Birkhoff-James orthogonality coincides only in Hilbert spaces.

\begin{theorem}
	Let $\mathbb{X}$ be a real Banach space and let $dim \, \mathbb{X} \geq 3.$ Then $\mathbb{X}$ is a Hilbert space if and only if there exists a  $ T \in \mathbb{L}(\mathbb{X}, \mathbb{X}^*)$ such that $\perp_T = \perp_B.$ 
\end{theorem}

\begin{proof}
	To prove the necessary part we assume $\mathbb{X}$ is a real Hilbert space.  Let us consider the map $T : \mathbb{X} \longrightarrow \mathbb{X}^* $ defined by $ Tx(y) = \langle y, x \rangle$, for each $ y \in \mathbb{X}.$ Then clearly $T$ is a  bounded linear operator. Clearly, $ x \perp_T y \iff x\perp_B y, $ for all $x, y \in \mathbb{X}.$ Therefore, $\perp_T = \perp_B.$\\
	Let us now prove the sufficient part of the theorem. Suppose that $ u,v,w \in \mathbb{X}$ be  such that $ u \perp_B w$ and $ v \perp_B w.$ Since $\perp_T = \perp_B,$  we have $ u \perp_T w$ and $ v \perp_T w.$  Then by left additivity of $\perp_T$, we get  $ u+v \perp_T w$ and so  $ u+v \perp_B T.$ This implies that $ \perp_B$ is left additive. Therefore, from \cite[Th. 2]{Jb} we conclude that $ \mathbb{X}$ is a Hilbert space. 
\end{proof}

\begin{remark} Analogous characterization for Hilbert spaces follows easily:  
	Let $\mathbb{X}$ be a complex Banach space and let $dim \, \mathbb{X} \geq 3.$ Then $\mathbb{X}$ is a Hilbert space if and only if there exists a conjugate linear operator $T : \mathbb{X} \to \mathbb{X}^*$  such that $\perp_T = \perp_B.$ 
\end{remark}

Our next result is characterization of two-dimensional  real Euclidean spaces out of all $\ell_p^2(\mathbb{R}) $ spaces. Before we note the following lemma which will be useful for our next theorem. 
\begin{lemma} \label{lem}
	Let $\mathbb{X} = \ell_p^2(\mathbb{R}),$ where $1 \leq p \leq \infty.$ Then $p=2$  If and only if $ (\alpha, \beta) \perp_B (\beta, -\alpha), $ for all $ (\alpha, \beta) \in \mathbb{X}.$ 
\end{lemma}

\begin{theorem}
	Let $ \mathbb{X} = \ell_p^2(\mathbb{R}).$ Then $p=2$  if and only if there exists an operator  $ T \in \mathbb{L}(\mathbb{X}, \mathbb{X}^*) $ such that $ \perp_T = \perp_B.$ 
\end{theorem}

\begin{proof}
	Since the necessary part is immediate, we only prove the sufficient part. Let $ e_1, e_2 \in S_{\ell_p^2},$ where $ e_1 = (1, 0)$ and $ e_2 = (0, 1).$ We note that $ e_1 \perp_B e_2,$ $ e_2 \perp_B e_1$ and $ (e_1 + e_2) \perp_B( e_1 - e_2).$ Let $ e_1^*, e_2^* \in \ell_q^2$ be such that $ e_i^*(e_j)= \delta_{ij},$ where $ i \in \{1, 2\}$ and $ \frac{1}{p} + \frac{1}{q} = 1.$ Then we write $Te_1 = ae_1^* + be_2^*$ and $ Te_2 = ce_1^* + de_2^*,$ for some $ a, b, c, d \in \mathbb{R}.$ Since $ \perp_T = \perp_B,$ it is easy to observe that $ Te_1 = ae_1^*$ and $ Te_2 = de_2^*.$ Now using $ (e_1 +e_2) \perp_T (e_1 - e_2),$ we obtain $a = d.$ So $ Te_1 = ae_1^*$ and $ Te_2 = ae_2^*.$ Now for any $ (\alpha, \beta) \in \mathbb{X},$ $ (T(\alpha e_1 + \beta e_2), (\beta e_1 - \alpha e_2)) = 0.$ This implies that $ (\alpha e_1 + \beta e_2) \perp_B (\beta e_1 - \alpha e_2),$ i.e., $(\alpha, \beta) \perp_B (\beta, -\alpha),$ for all $(\alpha, \beta) \in \ell_p^2.$ Then from Lemma \ref{lem}, we have $p=2.$ This completes the proof of the theorem.
\end{proof}

\begin{theorem}
	Let $ \mathbb{X}$ be a two-dimensional real Banach space. Then $ \mathbb{X}$ is a Hilbert space if and only if the following conditions hold true:
	\begin{enumerate}[label=\upshape(\roman*), leftmargin=*, widest=iii]
    \item There exists a $ T \in \mathbb{L}(\mathbb{X}, \mathbb{X}^*)$ such that $ \perp_T = \perp_B$
	\item There exists $u, v \in \mathbb{S}_\mathbb{X}$ such that $ u \perp_B v,$ $ v \perp_B u,$ $ (u+v) \perp_B (u-v)$ and $ \| (\gamma + \kappa\delta)u + (\delta - \kappa\gamma)v\| = \|  (\gamma - \kappa\delta)u + (\delta + \kappa\gamma)v\|,$ for all $ \gamma, \delta, \kappa \in \mathbb{R}.$	
	\end{enumerate}

\end{theorem}
\begin{proof}
	Clearly, if $ \mathbb{X}$ is Hilbert space then $ (i)$ and $(ii)$ holds. We only prove the sufficient part of the theorem. Since $ \perp_T = \perp_B,$ it is easy to note that $ Tu = \alpha u^*$ and $ Tv = \beta v^*,$ where $ \alpha, \beta \in \mathbb{R}$ and $u^*(u) = 1, v^*(v) = 1.$ As $  (u+v) \perp_B (u-v),$ we have $ (T(u+v), (u - v)) = 0,$ which gives us $ \alpha = \beta.$ Now for any $ \gamma, \delta \in \mathbb{R},$ it can be seen that $ (T(\gamma u + \delta v), (\delta u - \gamma v)) = 0. $ From Proposition \ref{basic}(v) we observe that $ \mathbb{X}$ is smooth and therefore, $(\gamma u + \delta v)^{\perp} = \kappa(\delta u - \gamma v),$ i.e., $ (\gamma u + \delta v) \perp_B \kappa(\delta u - \gamma v). $ Also we observe that $ \|(\gamma u + \delta v) + \kappa(\delta u - \gamma v)\| = \|(\gamma + \kappa\delta)u + (\delta - \kappa\gamma)v\| =  \|  (\gamma - \kappa\delta)u + (\delta + \kappa\gamma)v\| = \|(\gamma u + \delta v) - \kappa(\delta u - \gamma v)\|. $ This implies that $(\gamma u + \delta v) \perp_I \kappa(\delta u - \gamma v). $ Thus we get $ \perp_B \implies \perp_I.$ Therefore,  from \cite[Chapter 4]{A},  it can be concluded that $ \mathbb{X}$ is Hilbert space.
\end{proof}

\begin{theorem}
	Let $\mathbb{H}$ be a complex Hilbert space. Then there exists no $ T \in \mathbb{L}(\mathbb{H}, \mathbb{H}^*)$ such that $ \perp_T = \perp_B.$
\end{theorem}

\begin{proof}
	Suppose on the contrary that there exists $ T \in \mathbb{L}(\mathbb{H}, \mathbb{H}^*)$ such that $\perp_T = \perp_B.$ For any nonzero $ x\in \mathbb{H},$ we have $ x \perp_T y \iff x \perp_B y,$ for all  $ y \in \mathbb{H}.$ This implies that $ (Tx, y) = 0 \iff x^*(y) = 0,$ where $ x^*$ is the support functional of $x.$ Therefore, one can see that $ ker Tx = ker~ x^*.$ This gives us $ Tx = \lambda_x x^*,$ for some $ \lambda_x \in \mathbb{C}.$ One can easily observe that $T$ is bijective and therefore, $\lambda_x \neq 0.$ Now choosing an $\alpha \in \mathbb{C}$ with nonzero imaginary part, it can be clearly seen that $ T(\alpha x) = \bar{\alpha} Tx.$ Thus we get $ T$ is not linear. This contradicts $ T \in \mathbb{L}(\mathbb{H}, \mathbb{H}^*).$ Hence the theorem.   
\end{proof}

\subsection*{Acknowledgements}
The research of Dr. Debmalya Sain is supported by grant PID2021-122126NB-C31 funded by MCIN/AEI/10.13039/501100011033 and by ``ERDF A way of making Europe". The second author would like to thank  CSIR, Govt. of India, for the financial support in the form of Junior Research Fellowship under the mentorship of Prof. Kallol Paul.


\end{document}